\documentclass[12pt]{amsart} 
\usepackage{amsmath,amssymb,amsthm,mathrsfs}
\usepackage[english]{babel}
\usepackage[utf8]{inputenc}

\usepackage{color}
\usepackage{graphicx}
\def\r{\mathbb R}
\def\t{\mathsf T}

\newtheorem{theorem}{Theorem}[section]
\newtheorem{proposition}[theorem]{Proposition}

 \theoremstyle{definition}

\newtheorem{remark}[theorem]{Remark}

\begin{document}

\title{ Invariant $\lambda$-translators for the Gauss curvature flow in Euclidean space}
\author{Muhittin Evren Aydin}
\address{Department of Mathematics, Faculty of Science, Firat University. 23200 Elazig, Turkey.}
\email{meaydin@firat.edu.tr}
\address{ Departamento de Geometr\'{\i}a y Topolog\'{\i}a\\  Universidad de Granada. 18071 Granada, Spain}
\email{rcamino@ugr.es}
\author{Rafael L\'opez}
\begin{abstract}
A $\lambda$-translator is a surface in Euclidean space $\r^3$ whose Gauss curvature $K$ satisfies   $K=\langle N, \vec{v} \rangle +\lambda$, where   $N$ is the Gauss map, $\vec{v}$ is a fixed direction, and $\lambda \in \r$. In this paper, we classify all $\lambda$-translators that are invariant by a one-parameter group of translations and a one-parameter group of rotations. 
 \end{abstract}

\maketitle

\section{Introduction}

Let $\vec{v}$ be a unit vector in Euclidean space $\r^3$, and $\lambda\in\r$. A surface $\Sigma$ immersed in $\r^3$ is said to be a {\it $\lambda$-translator} for the Gauss curvature flow if its Gauss curvature $K$ satisfies 
\begin{equation}\label{eq1}
K=\langle N, \vec{v} \rangle +\lambda, \quad \vec{v} \in \r^3, \quad |\vec{v}|=1, 
\end{equation}
where   $N$ is the unit normal vector of $\Sigma$. The vector $\vec{v}$ is called the  {\it speed} of the flow.  For simplicity, we will refer to such surfaces as $\lambda$-translators. A  $\lambda$-translator with $\lambda =0$ is  called a translator. 
 When $\lambda=0$, these translators   correspond to solitons of the flow by Gauss curvature. In this situation, the surface $\Sigma$ evolves under the Gauss curvature flow purely by translation in a fixed direction $\vec{v}$. That is, the family of surfaces $\{\Sigma+t\vec{v}:t\in \r \}$ satisfies the property that, at each fixed $t$, the normal component of the velocity vector $\vec{v}$ at each point is equal to the Gauss curvature at that point.  

The study of the flow by Gauss curvature  goes back to the work of Firey to model the shape of  tumbling stones \cite{fi}. The generalization in    higher dimensions is due to Tso \cite{ts}. See also \cite{an,an3,cdl,ccd,ur}. Later, this was   generalized to the flow by the powers of the Gauss curvature as discussed in \cite{ch,ur2,ur3}, and further improved in \cite{al,al2,chu,cdl,cdkl,cck,jbj}. 

Equation \eqref{eq1} is invariant by rigid motions of $\r^3$. If the orientation on $\Sigma$ is reversed, then $\Sigma$ becomes a $\lambda$-translator with speed $-\vec{v}$. Hence, fixed the speed $\vec{v}$, the sign of $\lambda$ cannot be prescribed. Examples of $\lambda$-translators are the following: 

\begin{enumerate}
\item Planes. Given $\vec{v}$, any plane is a $\lambda$-translator with speed $\vec{v}$ if $\lambda=-\langle N,\vec{v}\rangle$.
\item Circular cylinders whose rotation axis is parallel to $\vec{v}$. Here $\lambda=0$ because $\langle N,\vec{v}\rangle=0$ and $K=0$.
\item Rotational cones. If $\Sigma$ is a rotational cone whose rotation axis is parallel to $\vec{v}$, and since $K=0$, then $\Sigma$ is a $\lambda$-translator, where $\lambda=-\cos\theta$, and $\theta$ is the inner angle of the cone.
\end{enumerate}

In this paper we will investigate $\lambda$-translators that are invariant by two types of one-parameter groups of rigid motions of $\r^3$: the group of translations determined by a vector $\vec{w}$, and the group of rotations determined by a fixed axis $\vec{w}$. The objective of this paper is the classification of all $\lambda$-translators within these two classes of surfaces. Notice that in both of the above cases, there is not a {\it priori} relationship between the vector $\vec{w}$ and the speed $\vec{v}$ in the definition of a $\lambda$-translator. However, in either case, we will show that $\vec{w}$ must be parallel to $\vec{v}$ (Thm. \ref{pr-cyl} and Prop. \ref{pr-ax}).

On the one hand, a cylindrical surface in $\r^3$ is defined as a surface  invariant by a one-parameter group of translations. In Thm. \ref{pr-cyl}, we provide a complete classification of cylindrical $\lambda$-translators. On the other hand, cylindrical surfaces belong to two well-known classes of surfaces: ruled surfaces and translation surfaces. In Sect. \ref{s2}, we also study $\lambda$-translators of ruled and translation types. We classify in Thm. \ref{pr-rul} the  non-cylindrical ruled $\lambda$-translator  showing that the Gauss curvature satisfies $K=0$ and that the ruled surface is a tangent surface whose base curve is a general helix. Translation $\lambda$-translators are classified in   Thm \ref{pr-tra} proving that they must be   cylindrical surfaces.

In Sects. \ref{s3} and \ref{s4}, we study $\lambda$-translators that are invariant by a one-parameter group of rotations, providing a complete classification in Thms. \ref{t1}, \ref{t2}, and \ref{t5}. As an immediate observation in Sect. \ref{s3}, and as illustrated by the above examples of $\lambda$-translators, we prove in Prop. \ref{pr-str} that all rotational surfaces generated by straight lines are $\lambda$-translators. Among the examples in our classification, and depending on the range of $\lambda$, Thms. \ref{t1} and \ref{t2} establish the existence of rotational $\lambda$-translators that intersect the rotation axis orthogonally. The construction of convex entire graphs is also included in Thm. \ref{t2}. In contrast, as a consequence of Thm. \ref{t2}, we observe the non-existence of closed rotational $\lambda$-translators. 
For closed $\lambda$-translators, we prove in Thm. \ref{t-gb}   that $\lambda \geq 0$ and     the genus is $0$ ($\lambda>0$) or $1$ ($\lambda=0$).  

\section{Cylindrical $\lambda$-translators}\label{s2}

Cylindrical surfaces belong to a wider class of surfaces called ruled surfaces. A ruled surface in $\r^3$ is constructed by moving a straight-line along a given curve. The straight-line and the curve are called the ruling and the base curve of the surface, respectively.  If the rulings are all parallel straight-lines, then the ruled surface becomes cylindrical. A parametrization of a ruled surface $\Sigma$ is 
\begin{equation}
\Psi(s,t)=\gamma(s)+tw(s), \quad s\in I \subset \r,t\in\r, \label{rule}
\end{equation}
where $ \gamma \colon I \to \mathbb{R}^3$,  $\gamma = \gamma(s)$, is the base curve parametrized by arc-length and
$w(s)$ is a unit vector field along $\gamma$ which indicates  the direction of rulings of $\Sigma$.

First, we consider the cylindrical $\lambda$-translators.  In this case, the vector field $w(s)$ in the parametrization \eqref{rule} is a fixed vector $\vec{w}$. Then $\Sigma$ is invariant by the group of translations defined by $\vec{w}$ and   the base curve $\gamma$ can be chosen to lie in a plane orthogonal to $\vec{w}$. The following result classifies cylindrical $\lambda$-translators.

\begin{theorem} \label{pr-cyl}
Cylindrical $\lambda$-translators with speed $\vec{v}$ are either planes orthogonal to $\vec{v}$ ($\lambda =\pm 1$) or translators whose rulings are parallel to $\vec{v}$.
\end{theorem}

\begin{proof}
Let $\vec{v}\in \r^3$ be a unit vector, $|\vec{v}|=1$. Any plane orthogonal to $\vec{v}$ is a $(\pm 1)$-translator, meaning it satisfies Eq. \eqref{eq1} with $\lambda=\pm 1$. This proves the first part of the theorem.

Suppose now that $\Sigma$ is a non-planar cylindrical surface and we will prove that $\lambda=0$. In such a case, we will also establish that the rulings are parallel to $\vec{v}$. By contradiction, suppose $\lambda\not=0$. The surface $\Sigma$ is given by  $\Sigma=\{\gamma(s)+t\vec{w}:t\in \r,s\in I \}$, where $\gamma\colon I\subset\r\to\r^3$ is parametrized by arc-length and lies in a plane orthogonal to $\vec{w}$. Since $K=0$ for any cylindrical surface, Eq.  \eqref{eq1} can be written as
\begin{equation}\label{nn}
\langle {\bf n}(s),\vec{v}\rangle+\lambda=0,
\end{equation}
for all $s\in I$, where ${\bf n}$ is the unit normal vector of $\gamma$ as planar curve. Differentiating \eqref{nn} with respect to $s$, we have
$\kappa(s)\langle\gamma'(s),\vec{v}\rangle=0$ for all $s\in I$. Since $\Sigma$ is non-planar, $\gamma$ is not a straight-line, and therefore we have $\kappa\not=0$. We then deduce $\langle\gamma'(s),\vec{v}\rangle=0$ for all $s\in I$; that is, $\gamma$ lies in a plane orthogonal to $\vec{v}$. This proves that $\vec{v}$ and $\vec{w}$ are parallel. As a consequence ${\bf n}(s)$ is orthogonal to $\vec{v}$ and Eq. \eqref{nn} implies $\lambda=0$.

\end{proof}

We now classify the non-cylindrical ruled $\lambda$-translators. In the parametrization \eqref{rule} we can  assume that $w'(s) \neq 0$ on $I$. Moreover, we can choose the base curve  $\gamma$ to be the striction line which implies that   $\langle \gamma'(s),w'(s)\rangle=0$, for every $s\in I$. The classification is the following.  

\begin{theorem} \label{pr-rul}
The only non-cylindrical ruled $\lambda$-translators with speed $\vec{v}$ are the tangent surfaces whose base curves are general helices with axis $\vec{v}$. Moreover, the Gauss curvature satisfies $K=0$, $\lambda \neq 0$, and $\lambda \in [-1,1]$.
\end{theorem}

\begin{proof}
Let $\Sigma$ be a ruled surface which satisfies Eq. \eqref{eq1}, and assume that $\Sigma$ is non-cylindrical.  Then, $\Sigma$ admits the parametrization given by \eqref{rule}, with the condition $\langle \gamma'(s),w'(s)\rangle=0$, for every $s\in I$. We set
$$
\alpha=\frac{(\gamma',w,w')}{|w'|^2},
$$
where $(\gamma',w,w')$ is the determinant of the vectors $\gamma'$, $w$, and $w'$. A direct computation yields
$$
K=-\frac{\alpha^2}{(\alpha^2+t^2)^2}, \quad N=\frac{\alpha w'+tw'\times w}{|w'|\sqrt{\alpha^2+t^2}}.
$$
We consider two separate cases:
\begin{enumerate}
\item  Case $\alpha=0$. In this case $\gamma'$ is parallel to $w$, because $w'$ is orthogonal to both $\gamma'$ and $w$. Hence, $\Sigma$ is a tangent surface. Let ${\bf n}$ and ${\bf b}$ denote the principal normal and binormal vector fields along $\gamma$, respectively. Then, $N=-{\bf b}$ and Eq. \eqref{eq1} reduces to $\langle {\bf b} , \vec{v} \rangle =\lambda$, where $|\lambda| \leq 1$. If $\lambda=0$, then $\Sigma$ would be a plane, which contradicts our assumption. Differentiating both sides with respect to $s$, we obtain $\tau\langle {\bf n} , \vec{v} \rangle =0$, where $\tau$ is the torsion of $\gamma$. Note that $\tau \neq 0$, because otherwise $\Sigma$ would again be planar. Consequently, $\langle {\bf n} , \vec{v} \rangle =0$, which means that $\gamma$ is a general helix with axis $\vec{v}$. 

\item  Case $\alpha\neq 0$. According to \cite[Theorem 6]{al}, a ruled surface with $K\neq 0$ does not satisfy Eq. \eqref{eq1} when $\lambda = 0$. Hence, by assuming $\lambda \neq 0$, we will arrive at a contradiction. Then, Eq.   \eqref{eq1} can be written as
$$
-\frac{\alpha^2+\lambda(\alpha^2+t^2)^2}{(\alpha^2+t^2)^2}=\frac{\alpha \langle w', \vec{v} \rangle+t (w',w,\vec{v})}{|w'|\sqrt{\alpha^2+t^2}}.
$$
By squaring both sides of this equation and passing to one side,  we obtain 
$$
\sum_{n=0}^8A_n(s)t^n=0,
$$
where $A_n(s)$ are smooth functions on the variable $s$. Looking this equation as a polynomial equation on the variable $t$, we deduce that the coefficients $A_n$ must all vanish. A computation of $A_8$ gives
$$
\lambda^2-\frac{(w',w,\vec{v})^2}{|w'|^2}=0.
$$
Since $\lambda \neq 0$, it follows that $(w',w,\vec{v})\neq 0$. Using $A_7=0$, we obtain
$$
2\alpha\frac{\langle w', \vec{v} \rangle (w',w,\vec{v})}{|w'|^2}=0.
$$
Thus we deduce $\langle w', \vec{v} \rangle=0$. Finally, from $A_2=0$ and $A_4=0$ we derive
$$
\lambda\alpha^4(4+3\lambda\alpha^2)=0, \quad \lambda\alpha^2(2+3\lambda \alpha^2)=0,
$$
which leads to a contradiction.
\end{enumerate}
\end{proof}

\begin{remark}
It was proved in Prop. 2.4 that non-cylindrical ruled $\lambda$-translators have zero Gauss curvature. In this case, Eq.   \eqref{eq1} becomes $\langle N, \vec{v} \rangle = -\lambda$. Therefore, the angle between the unit normal vector of the surface and the direction $\vec{v}$ is constant. This implies that the surface belongs to the class of constant angle surfaces, a topic that has been extensively studied in differential geometry (see, for example, \cite{cd, dfvv, dsr, mn}). Recently, it was shown in \cite{loy} that a constant angle surface is the tangent surface on a general helix, which aligns with Prop. 2.4.
\end{remark}

The second family of $\lambda$-translators we study consists of surfaces that can be expressed in separable  variables in the form $z=f(x)+g(y)$, where  $f:I\subset \r \to \r$ and $g:J\subset \r \to \r$ are smooth functions. A parametrization of the surface is
 $$
\Psi(x,y)=(x,y,f(x)+g(y)) 
$$ 
and the surface is called  a {\it translation surface}. Notice that the surface is obtained by the translation of the curve $x\mapsto (x,0,f(x))$ along the curve $y\mapsto(0,y,g(y))$ and vice-versa. In the following result, we classify all $\lambda$-translators that are of translation type. We point out that translation surfaces are graphs $z=f(x)+g(y)$ on the $xy$-plane and there is no a {\it priori} relation between the $xy$-plane and the speed $\vec{v}$ in the definition of a $\lambda$-translator. 

\begin{theorem} \label{pr-tra}
Cylindrical surfaces are the only $\lambda$-translators of the form $z=f(x)+g(y)$.
\end{theorem}

\begin{proof}
Let $\Sigma$ be a $\lambda$-translator of the form $z=f(x)+g(y)$.  With respect to the $(x,y,z)$ coordinates, let   $\vec{v}=(v_1,v_2,v_3)$ be the speed of flow. If $\lambda=0$, the result was proved in \cite[Theorem 5]{al}. 

We now consider the case $\lambda\not=0$. Assume, by contradiction, that $\Sigma$ is not a cylindrical surface. Then there exists $(x_0,y_0)\in I\times J$ such that $f''(x_0)\not=0$ and $g''(y_0)\not=0$. Therefore, we can assume $f''g''\not=0$ on a subset $\tilde{I} \times \tilde{J}$, where $\widetilde{I}$ and $\widetilde{J}$ are subintervals of $I$ and $J$ containing $x_0$ and $y_0$, respectively.  We restrict our study to $\tilde{I} \times \tilde{J}$. Eq. \eqref{eq1} can be expressed by 
\begin{equation}
f''g''W_1^{-2}=W_1^{-1/2}W_2+\lambda, \label{fg}
\end{equation}
where $W_1=1+f'^2+g'^2$ and $W_2=-v_1f'-v_2g'+v_3$. We multiply both sides of Eq. \eqref{fg} by $W_1^2/g''$, obtaining
$$
f''=\frac{1}{g''}(W_1^{3/2}W_2+\lambda W_1^2).
$$
Differentiating with respect to $y$,
\begin{equation}
0=-\frac{g'''}{g''^2}(W_1^{3/2}W_2+\lambda W_1^2)+\frac{1}{g''}(3g'g''W_1^{1/2}W_2-v_2g''W_1^{3/2}+4g'g''\lambda W_1).  \label{fg-1}
\end{equation}
This equation can be written in the form $A+B \sqrt{W_1}=0$, where $A$ and $B$ are functions of $g,g',g'',W_1,$ and $W_2$. The identity $ B^2 W_1-A^2=0$ is a polynomial in $f'$ of degree $8$, namely,  
\begin{equation}\label{final}
\sum_{n=0}^8C_n(y)f'(x)^n=0.
\end{equation}

Therefore, all coefficients $C_n$ must vanish. By the computations of $C_6,C_7$, and $C_8$, we obtain 
\begin{eqnarray*}
C_6&=&v_2^2-2\frac{g'g'''}{g''^2}(3v_1^2-4\lambda^2), \\
C_7&=& -2v_1v_2\frac{g'''}{g''^2}, \\
C_8&=&\left (\frac{g'''}{g''^2} \right)^2(v_1^2-\lambda^2)
\end{eqnarray*}
We distinguish two cases.
 \begin{enumerate}
 \item Case $g'''\not=0$ at some point $y_0$. We work around a subinterval $\widetilde{J}$ around $y_0$. Then, from $C_8=0$, we get $v_1^2=\lambda^2$; in particular this implies $v_1 \neq 0$. Hence, $C_7=0$ yields $v_2=0$. Substituting this in $C_6=0$ leads to the contradiction 
$$2\lambda^2\frac{g'g'''}{g''^2}=0.$$
\item Case $g'''=0$ identically. From $C_6=0$ we have $v_2=0$. Thus, for $5\leq n\leq 8$, the coefficients $C_n$ vanish trivially. Moreover, $C_4=g'^2(9v_1^2-16\lambda^2)$, which implies $v_1 \neq 0$; otherwise, from $C_4=0$, we would conclude $\lambda =0$. From the identity $C_3=18g'^2v_1v_3$, it follows $v_3=0$. Finally, we arrive at the contradiction $C_0=-16\lambda^2(1+g'^2)^2=0$.
 \end{enumerate}
\end{proof}

\section{Rotational  $\lambda$-translators} \label{s3}
 
In the next two sections, we will study the $\lambda$-translators of rotational type, obtaining geometric properties of these surfaces. In this section, we   study the relationship between the rotation axis and the speed $\vec{v}$ of the  $\lambda$-translator and we also consider the problem of existence of rotational $\lambda$-translators intersecting orthogonally the rotation axis.

The first result proves that    the rotation axis and $\vec{v}$ are parallel.

\begin{proposition} \label{pr-ax}
Let $\Sigma$ be a   rotational surface. If $\Sigma$ is a $\lambda$-translator, then $\Sigma$ is a plane or its rotation axis is parallel to the speed $\vec{v}$.
\end{proposition}

\begin{proof}
 After a rigid motion of $\r^3$, we   can assume that the rotation axis $L$ coincides with the $z$-axis. Let $\gamma\colon I\to\r^3$ be the generating curve of $\Sigma$ which we assume to be parametrized by arc-length,  $\gamma(s)=(x(s),0,z(s))$, $s\in I$, $x(s)>0$. The surface $\Sigma$ can be parametrized by  
\begin{equation}\label{para}
\Psi(s,t)=(x(s)\cos t,x(s)\sin t,z(s)), \quad s\in I, t\in \r.
\end{equation}
 The unit normal vector of $\Sigma$ is $N=(-z'\cos t,-z'\sin t,x')$, and the Gauss curvature   is $K=\frac{z'}{x}\kappa$, where $\kappa$ is the curvature of $\gamma$. If the coordinates of  the speed $\vec{v}$ with respect to $(x,y,z)$ are $\vec{v}=(v_1,v_2,v_3)$, then Eq. \eqref{eq1}   becomes
\begin{equation*}
-v_1z'(s)\cos t-v_2z'(s)\sin t+v_3x'(s)-\frac{z'(s)}{x(s)}\kappa(s)+\lambda=0.
\end{equation*}
Since the functions $\{1,\sin t,\cos t\}$ are linearly independent, their coefficients, which are functions on the $s$ variable, must vanish identically in the interval $I$. Therefore  
\begin{equation}
\begin{split}
0&=v_1z', \\
0&=v_2z', \\
0&=v_3x'-\frac{z'}{x}\kappa+\lambda.
\end{split}  \label{kap}
\end{equation}
We distinguish two cases.
\begin{enumerate}
\item There exists $s_1\in I$ such that $z'(s_1)\not=0$. Then the first two equations of \eqref{kap} imply $v_1=v_2=0$. This gives $\vec{v}=(0,0,\pm 1)$. Thus $\vec{v}$ is parallel to the $z$-axis,  proving the result.
\item Case $z'=0$ identically in $I$. Then $\gamma$ is a horizontal line. This proves that $\Sigma$ is a horizontal plane, completing the proof.  
\end{enumerate}
\end{proof}

 From now, we will assume that the speed of the Gauss curvature flow is $\vec{v}=(0,0,1)$ and in consequence the rotation axis is the $z$-axis.  Following with the notation of Prop. \ref{pr-ax}, since    $\gamma$ is parametrized by arc-length, there is a smooth function $\theta=\theta(s)$ such that   $x'(s)=\cos\theta(s)$ and  $z'(s)=\sin\theta(s)$. In particular, $\kappa(s)=\theta'(s)$. From   the last equation of \eqref{kap}, a rotational surface $\Sigma$ parametrized by \eqref{para}  is a $\lambda$-translator   if and only if the functions $x(s)$ and $z(s)$ satisfy
\begin{equation}\label{trig}
\left\{
\begin{split}
x'(s)&=\cos\theta(s) \\ 
z'(s)&=\sin\theta(s) \\
\theta'(s) &=x(s)\frac{\cos\theta(s)}{\sin\theta(s)}+\lambda \frac{x(s)}{\sin\theta (s)}.
\end{split}\right.
\end{equation}
 
In the next result, we investigate the  rotational $\lambda$-translators  whose generating curve is  a straight-line.

\begin{proposition} \label{pr-str}
Any rotational surface generated by straight lines is a $\lambda$-translator. In particular, they are: horizontal planes ($|\lambda|=1$), rotational cones ($0<|\lambda|<1$) and circular cylinders ($\lambda=0$).
\end{proposition}

\begin{proof} Let $\Sigma$ be a rotational surface about the $z$-axis and whose generating curve is a straight-line $\gamma$. Since $\theta'=\kappa=0$, the function $\theta$ is   constant, $\theta(s)=\theta_0$. If $\gamma$ is a horizontal straight-line, then $\Sigma$ is a plane and we know that $\Sigma$ is a $\lambda$-translator with $\lambda=\pm 1$. If $\gamma$ is not a horizontal plane, then  $\sin\theta_0\not=0$. Then \eqref{trig} proves that   $\Sigma$ is a $\lambda$-translator for the choice of $\lambda$ given by  $\lambda=-\cos\theta_0$. 
 \end{proof}

We point out that planes and circular cylinders have constant mean curvature. In \cite{lo3}, the second author has proved that these surfaces are the only $\lambda$-translators for the Gauss curvature flow with constant mean curvature.  
 
 In the second part of this section, we investigate those rotational $\lambda$-translators that intersect the rotation axis orthogonally.    If the intersection with the $z$-axis occurs at $s=0$, the orthogonality condition is equivalent to  the initial condition $\theta(0)=0$ in the system \eqref{trig}. Consequently, $\sin\theta(0)=0$ and $x(0)=0$, which implies that the third equation of \eqref{trig}  is not well-defined at $s=0$.

In order to prove the existence of \eqref{trig}, it is convenient to reparametrize the surface $\Sigma$. Since we require that $\gamma$ meets the $z$-axis orthogonally, $\gamma$ must be a graph on the $x$-axis.  By using radial coordinates for $\Sigma$, $x=r$ and $z=u(r)$, we have   $\gamma(r)=(r,0,u(r))$, and the parametrization of $\Sigma$ is $(r,t)\mapsto (r\cos t,r\sin t,u(r))$, where $u$ is defined on some interval of $(0,\infty)$. With this parametrization, Eq.  \eqref{trig} becomes:
\begin{equation}\label{r-01}
\dfrac{u'(r)u''(r)}{r(1+u'(r)^2)^{2}}=\dfrac{1}{\sqrt{%
1+u'^2(r)}}+\lambda. 
\end{equation}
This equation is singular at $r=0$ and standard  ODE theory cannot directly apply. The orthogonality condition with the $z$-axis is expressed by $u'(0)=0$. After a vertical translation, we can assume that the intersection point between  $\gamma$  and the $z$-axis occurs at the origin. To address the singularity at $r=0$ we multiply \eqref{r-01} by $2r$ and consider the initial value problem
\begin{equation} \label{r-0}
\left\{ 
\begin{split}
\left (\dfrac{u'(r)^2}{1+u'(r)^2}\right )'&=\frac{2r}{\sqrt{%
1+u'(r)^2}}+2\lambda r, \quad \mbox{in } (r_0,r_0+\delta) \\ 
u(r_0)&=0, \quad u'(r_0)=0,%
\end{split}
\right. 
\end{equation}
where $r_0>0$.    When $r_0\not=0$, the standard theory of ODEs assures the existence of solutions of \eqref{r-0} for some $\delta>0$. In the following result, we establish the existence of solutions of \eqref{r-0} in the case $r_0=0$, for some $\delta>0$.

\begin{theorem} \label{t1}
For $r_0=0$, the solvability of \eqref{r-0} for some $\delta>0$ is the following:
\begin{enumerate}
\item If
$\lambda<-1$, then there are no solutions.
\item If $\lambda=-1$, the solution is $u(r)=0$.
\item If $\lambda>-1$, there is $R>0$ such that the initial value problem \eqref{r-0} has a solution in the interval $[0,R]$. 
\end{enumerate}
\end{theorem}

\begin{proof}
Suppose that \eqref{r-0} has a solution for $r_0=0$. By taking the limit as $r\to 0$ and applying the L'H\^{o}pital's rule in \eqref{r-01}, we obtain
$$u''(0)^2= 1+\lambda.$$
This implies that $\lambda$ must satisfy the inequality  $\lambda\geq -1$. 

If $\lambda=-1$, then it is trivial that $u(r)=0$ is a solution. 
Therefore, it remains to prove the existence of solutions when $\lambda>-1$. Define the functions   $f\colon\r\to (-1,1)$ and $g:  \r \to  \r$ by
$$
f(x)=\frac{x}{\sqrt{1+x^2}}, \quad g(x)=\frac{2}{\sqrt{1+x^2}}+2\lambda.
$$
Note that the positivity of $g$ depends on the value of $\lambda$. For example, when $\lambda \geq 0$, $g$ is always positive. However, if $\lambda \in (-1,0)$, then $g$ remains positive for every $x \in (\frac{\sqrt{1-\lambda^2}}{\lambda},\frac{\sqrt{1-\lambda^2}}{-\lambda})$.

In terms of the functions $f$ and $g$, the equation in \eqref{r-0} can be written as $(f(u')^2)'=rg(u')$. From this identity   we arrive at
\begin{equation}\label{u3}
u(r)= \int_0^r f^{-1} \left (\sqrt{\int_0^s tg(u'(t))dt }\right)ds,
\end{equation}
where $f^{-1}$ is the inverse of $f$, defined by 
$$f^{-1}\colon (-1,1)\to\r, \quad f^{-1}(y)=\frac{y}{\sqrt{1-y^2}}.$$
Define the operator $\t$ by
$$
(\t u)(r)=\int_0^r f^{-1} \left (\sqrt{\int_0^s tg(u'(t))dt }\right)ds.
$$
It is clear that a $C^2$ function $u$ is a solution of \eqref{r-0} if $u$ is a fixed point of the operator $\t$. The proof of the theorem consists of showing that $\t$ is a contraction map on the Banach space $C^1([0,R])$   endowed with the standard norm $\Vert u \Vert = \Vert u \Vert_\infty + \Vert u' \Vert_\infty$. By applying the Banach fixed point theorem, the existence of a fixed point  guarantees the desired solution. For this, we will prove that $\t$ is a   contraction on a closed ball $\overline{\mathcal{B}(0,\epsilon )} \subset C^1([0,R])$, for some $\epsilon,R>0$, which   will be fixed later.    

Notice that as long as $\lambda \geq 0$, $\t$ is well-defined because the radicand in its definition is always positive. If $\lambda \in (-1,0)$ we can ensure that the radicand remains positive by choosing a suitable upper bound for $\epsilon$; for example we may require $\epsilon < \frac{\sqrt{1-\lambda^2}}{-\lambda}$. Therefore, $ |u' | <\epsilon$ because $\Vert u \Vert < \epsilon$.

Let $M=2+2\lambda>0$, which it is positive because $\lambda>-1$. Then  $g(x)\leq M$ for all $x\in\r$. Fix $R>0$ such that $R\leq\frac{1}{\sqrt{M}}$. Then we have
$$\int_0^stg(u'(t))\, dt\leq M\int_0^st\, dt=M\frac{s^2}{2}\leq M\frac{R^2}{2}\leq\frac12.$$
Thus, we can apply $f^{-1}$. 

We prove that $\t$ is a self-map on the closed ball $\overline{\mathcal{B}(0,\epsilon )}$ for some $\epsilon >0$. Second, we demonstrate that $\t$ is a contraction. Fix $\epsilon>0$, with $\epsilon<1$. In case that $\lambda\in (-1,0)$, we also assume $\epsilon < \frac{\sqrt{1-\lambda^2}}{-\lambda}$. The value of $R$ will be determined later.

\begin{enumerate}

\item We prove that $\t(\overline{\mathcal{B}(0,\epsilon )}) \subset \overline{\mathcal{B}(0,\epsilon )}$ where $R$ is given by  
\begin{equation}
R<\min  \{ \frac{1}{ \sqrt{M}},\frac{\epsilon}{2},\frac{\sqrt{2}\epsilon}{\sqrt{M(4+\epsilon^2)}}\}.  \label{RM-eps}
\end{equation}
Let $u\in  \overline{\mathcal{B}(0,\epsilon )}$. We know that $\int_0^s tg(u'(t))\, dt\leq 1/2$. By using that $f^{-1}$ is an increasing function,  we have
$$
|(\t u)(r)|\leq  \left| \int_0^r f^{-1} \left ( \frac{1}{\sqrt{2}}\right)\,ds \right|=\int_0^r ds=r \leq R\leq\frac{\epsilon}{2}.
$$
Similarly, we have $\int_0^s tg(u'(t))\, dt\leq \int_0^s Mt\, ds=Ms^2/2$. Thus 
\begin{equation*}
|(\t u)'(r)|\leq f^{-1}\left(\frac{s\sqrt{M}}{\sqrt{2}}\right)\leq f^{-1}\left(\frac{R\sqrt{M}}{\sqrt{2}}\right)\leq\frac{\epsilon}{2},
\end{equation*}
where we have used that $f^{-1}(R\sqrt{M/2})\leq\epsilon/2$ thanks to \eqref{RM-eps}.
Definitively, we have proved $\Vert \t u\Vert =|\Vert \t u\Vert_\infty+\Vert(\t u)'\Vert_\infty\leq\epsilon$.
 
\item We proceed in showing that $\t: \overline{\mathcal{B}(0,\epsilon )} \to \overline{\mathcal{B}(0,\epsilon )}$ is a contraction. For this, we aim to establish the existence of a positive constant $\mu$, with $\mu<1$, such that
$$
\Vert \t u - \t w \Vert \leq \mu \Vert  u - w \Vert ,
$$
for all $u,w \in \overline{\mathcal{B}(0,\epsilon )}$. Let $L_{f^{-1}}$,  $L_g $   denote the Lipschitz constants of $f^{-1}$ and $g$ respectively when these functions are defined in the interval $[-\epsilon,\epsilon]$. We are going to bound each of the two summands in  
$$
\Vert \t u - \t w \Vert = \Vert \t u - \t w \Vert_\infty + \Vert (\t u) - (\t w)' \Vert_\infty .
$$
Let $u,w\in \overline{\mathcal{B}(0,\epsilon )}\subset C^1([0,R])$.   For all $r \in [0,R]$, where $R$ will be changed later, we have
\begin{equation*}
\begin{split}
|(\t u) (r)- (\t w)(r)| &\leq L_{f^{-1}}  \int_0^r\left|\sqrt{ \int_0^s  t  g(u') \, dt}- \sqrt{ \int_0^s  t  g(w') \, dt}\right| ds  \\
&\leq L_{f^{-1}}\int_0^r \frac{\int_0^st|g(u')-g(w')|\, ds}{\sqrt{ \int_0^s  t  g(u') \, dt}+ \sqrt{ \int_0^s  t  g(w') \, dt}}\, ds   \\
&\leq L_{f^{-1}} L_g \Vert u-w  \Vert  \int_0^r \frac{s^2/2}{\sqrt{ \int_0^s  t  g(u') \, dt}+ \sqrt{ \int_0^s  t  g(w') \, dt}}\, ds.
\end{split}
\end{equation*}
Using the Taylor expansion of the function $g$, we have 
$$\int_0^s tg(u')\, dt=   (1+\lambda)s^2+o(s^2)$$
in $[-\epsilon,\epsilon]$. 
Thus, we have 
\begin{equation*}
\begin{split}
|(\t u) (r)- (\t w)(r)| & \leq L_{f^{-1}} L_g \Vert u-w  \Vert  \int_0^r \frac{s}{4\sqrt{1+\lambda}+o(1) }\, ds.
\end{split}
\end{equation*}
From \eqref{RM-eps}, we know that  $R<\epsilon/2$. By taking $\epsilon$ sufficiently small, there is a positive constant $C=C(\epsilon,\lambda)>0$ such that $4\sqrt{1+\lambda}+o(1)\geq C(\epsilon,\lambda)$. Then
\begin{equation}\label{ss1}
|(\t u) (r)- (\t w)(r)|  \leq \frac{1}{C} L_{f^{-1}} L_g \Vert u-w  \Vert \int_0^r s\, ds\leq \frac{R^2}{2C}L_{f^{-1}} L_g \Vert u-w  \Vert.
\end{equation}
Similarly, we derive 
\begin{equation}\label{ss2}
\begin{split}
|(\t u)' (r)- (\t w)'(r)| &\leq L_{f^{-1}}  \left|\sqrt{ \int_0^r  t  g(u') \, dt}- \sqrt{ \int_0^r  t  g(w') \, dt}\right|    \\
&= L_{f^{-1}}  \frac{\int_0^rt|g(u')-g(w')|\, ds}{\sqrt{ \int_0^r  t  g(u') \, dt}+ \sqrt{ \int_0^r  t  g(w') \, dt}}    \\
&\leq L_{f^{-1}} L_g \Vert u-w  \Vert    \frac{s^2/2}{\sqrt{ \int_0^s  t  g(u') \, dt}+ \sqrt{ \int_0^s  t  g(w') \, dt}}\\
&\leq\frac{R}{C}L_{f^{-1}} L_g \Vert u-w  \Vert.
\end{split}
\end{equation}
After choosing $R$ to satisfy the condition \eqref{RM-eps}, together with
$$
R\leq \min \left \{ \frac{\sqrt{C}}{\sqrt{2L_{f^{-1}} L_g}} , \frac{C}{4 L_{f^{-1}} L_g}  \right \},
$$
we obtain from \eqref{ss1} and \eqref{ss2}
$$
\Vert \t u - \t w \Vert \leq \frac12 \Vert  u - w \Vert ,
$$
which proves that $\t$ is a contraction on $C^1$([0,R]).
\end{enumerate}
To complete the proof, we show that the solution $u$ extends with $C^2$-regularity at $r=0$. This was proved at the beginning of the proof. Indeed, by taking the limit as $r\to 0$ and applying L'Hôpital's rule to Eq. \eqref{r-01}, we obtain
$$
u''(0)^2=1+\lambda,
$$
where $1+\lambda >0$ because $\lambda>-1$.  By the definition of $u$ in \eqref{u3},  the function $u$ implies that $u''(0)=\sqrt{1+\lambda}$.
\end{proof}

\begin{remark} \label{r1}
In \eqref{u3} we can also replace the inverse of $f^{-1}$ by $-f^{-1}$. In such a case, then  $u''(0)=-\sqrt{1+\lambda}$.  In fact,  it is not difficult to see that if $\bar{u}$ is the solution for Eq. \eqref{r-0} with $\bar{u}''(0)=-\sqrt{1+\lambda}$, then $\bar{u}(r)=-u(r)$ for all $r$, where $u$ is the solution with $u''(0)=\sqrt{1+\lambda}$.
\end{remark}
 

\section{Geometric properties of rotational $\lambda$-translators}\label{s4}

In this section we give a geometric description of the rotational $\lambda$-translators  that intersect orthogonally the rotation axis. In the final part of this section, we will also study the case that the surface does not meet the rotation axis. 

If the rotational $\lambda$-translator  meets orthogonally the rotation axis, we know by  Thm. \ref{t1} that  $\lambda\geq -1$. If $\lambda=-1$, the surface is a horizontal planes.  

 \begin{theorem}\label{t2}
Let $\lambda>-1$. Let $\Sigma$ be a rotational $\lambda$-translator with speed    $\vec{v}=(0,0,1)$ and parametrized by \eqref{para}. Suppose that $\gamma$ satisfies \eqref{trig} with initial conditions
\begin{equation}\label{trig2}
x(0)=z(0)=\theta(0)=0.
\end{equation}
Let   $[0,\omega)$ be the maximal domain of $\gamma$.  Up to a symmetry about the plane $z=0$, we have the following classification.
\begin{enumerate}
\item If $\lambda>0$  then $\omega<\infty$ and $\gamma$  intersects   the $z$-axis at $s=\omega$. This intersection is not orthogonal. Moreover,  $x(s)<\sqrt{2}$ for all $s\in I$.
\item If $\lambda=0$ then $\omega=\infty$. The curve $\gamma$ is asymptotic to the vertical line of equation $z=\sqrt{2}$. The function $z=z(s)$  is strictly increasing with $\lim_{s\to\infty}z(s)=\infty$.
\item If $\lambda\in (-1,0)$, then $\omega=\infty$. The curve $\gamma$ is an entire graph on $\r$ where $\lim_{s\to\infty}x(s)=\lim_{s\to\infty}z(s)=\infty$ and $\lim_{s\to\infty}\theta(s)=\cos^{-1}(-\lambda)$.   
 \end{enumerate}
In all cases, the surface is convex.
\end{theorem}

\begin{proof}

 By the L'H\^{o}pital rule we know that $\theta'(0)^2=1+\lambda$.  Then $\theta'(0)$ can take the value $\sqrt{1+\lambda}$ or $-\sqrt{1+\lambda}$. However we have seen in  Rem. \ref{r1} that the solutions for $\sqrt{1+\lambda}$ and $-\sqrt{1+\lambda}$ are symmetric with respect to each other about the $x$-axis.  Without loss of generality, we assume     $\theta'(0)=\sqrt{1+\lambda}$.  In such a case, the function $\theta(s)$ is increasing at $s=0$ and consequently, the functions $x(s)$ and $z(s)$ are also increasing at $s=0$.   

Multiplying the third equation of \eqref{trig} by $\sin\theta(s)$, we deduce a first integration. Indeed, it follows from \eqref{trig2} that
 \begin{equation}\label{fi}
1 -\cos\theta(s)=\frac12x(s)^2+\lambda\int_0^sx(t)\, dt.
\end{equation}
Notice that this identity can be also obtained from the first equation of \eqref{r-0}. 
\begin{enumerate}
\item Case $\lambda>0$.  We aim to prove that the function $\theta(s)$ crosses the value $\pi/2$. Assume that $\theta (s)\in (0,\pi/2)$; under this condition, the function $x(s)$ is increasing. Fix $\bar{s}>0$, chosen sufficiently close to zero. From \eqref{trig}, and provided $\theta(s)\in (0,\pi/2)$  and   $s>\bar{s}$, we have 
$$\theta'(s)\geq  \lambda\frac{x(s)}{\sin\theta(s)}\geq \lambda x(s) \geq\lambda x(\bar{s})>0.$$ This inequality proves that the growth of $\theta$ is linear when $\theta(s)\in (0,\pi/2)$. Consequently, there exists a first value $s_1$ such that $\theta(s_1)=\pi/2$. 

At $s=s_1$, we have $\theta(s_1)=\pi/2$, thus the function $x(s)$ attains a maximum at this point. Substituting into \eqref{fi}, we obtain
$$1=\frac12x(s_1)^2+\lambda\int_0^{s_1}x(t)\, dt>\frac12x(s_1)^2.$$
This inequality implies that $x(s_1)<\sqrt{2}$, and hence $x(s)\leq x(s_1)<\sqrt{2}$, for all $s\in [0,s_1]$.

It remains to prove that the solution curve $\gamma$  meets the rotation axis again. For this, we focus in the functions $(x,\theta)$ in the system \eqref{trig} and consider the autonomous system
\begin{equation}\label{au}
\left\{
\begin{split}
x'&=\cos\theta\\
\theta'&=\frac{x}{\sin\theta}(\cos\theta+\lambda).
\end{split}
\right.
\end{equation}
Here we are assuming $(x,\theta)\in \r\times (0,\pi)$ because of the denominator in $\theta'$. 
The singular point  of the system \eqref{au} is  $P=(0,\frac{\pi}{2})$.   The linearized system at the singular point $P=(0,\frac{\pi}{2})$ is 
$$\begin{pmatrix}0&-\sin\theta\\ \frac{\cos\theta+\lambda}{\sin\theta}&\frac{-1-\lambda\cos\theta}{\sin^2\theta}\end{pmatrix}(P)=\begin{pmatrix}0&-1 \\ \lambda &0\end{pmatrix}.$$
 Since the eigenvalues are pure complex numbers, namely $\pm\sqrt{-\lambda}$, then the point $P$ represents a center in the phase plane: see Fig. \ref{figphase}.  
\begin{figure}[hbtp]
\begin{center}
\includegraphics[width=.45\textwidth]{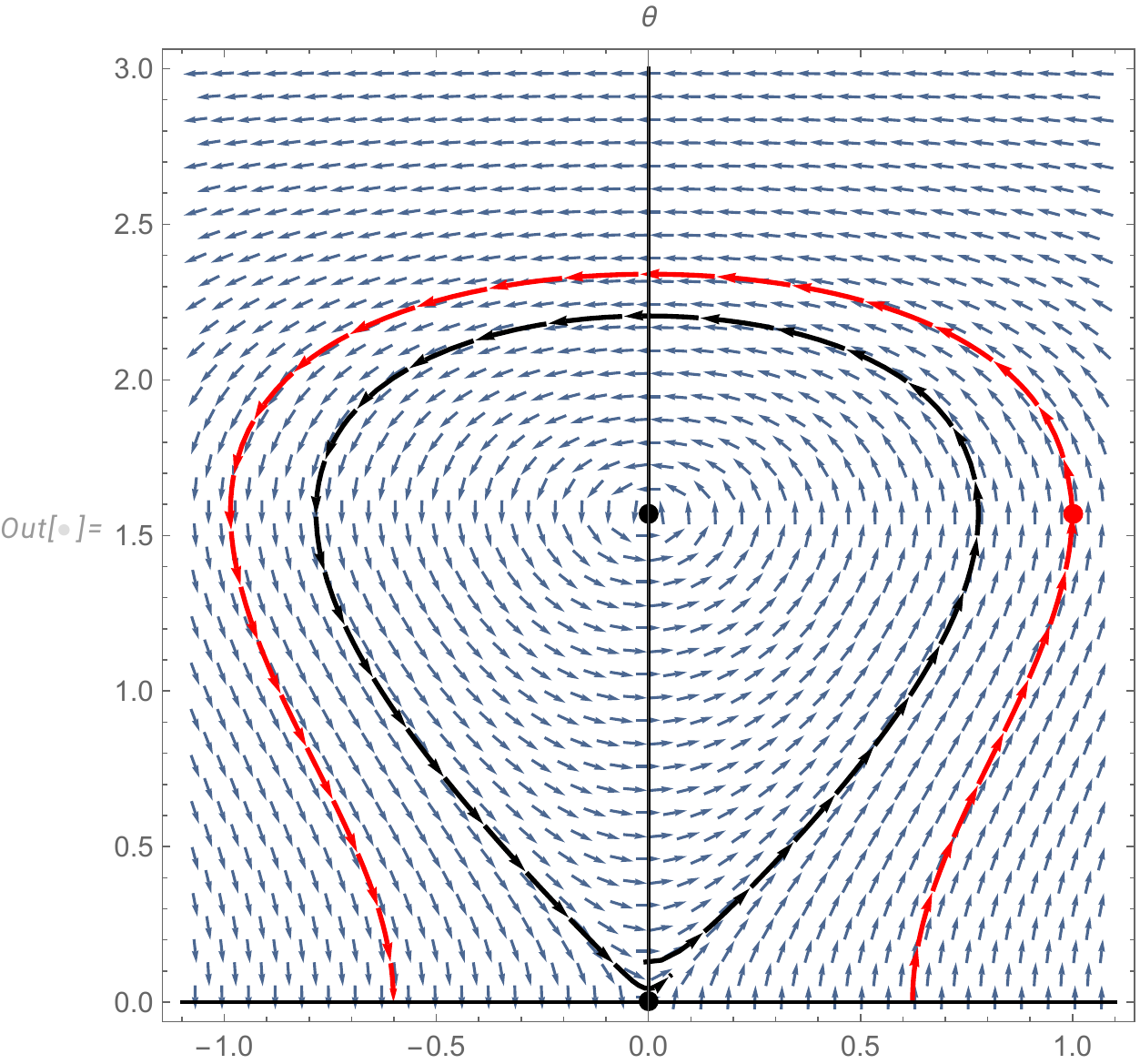}     \includegraphics[width=.45\textwidth]{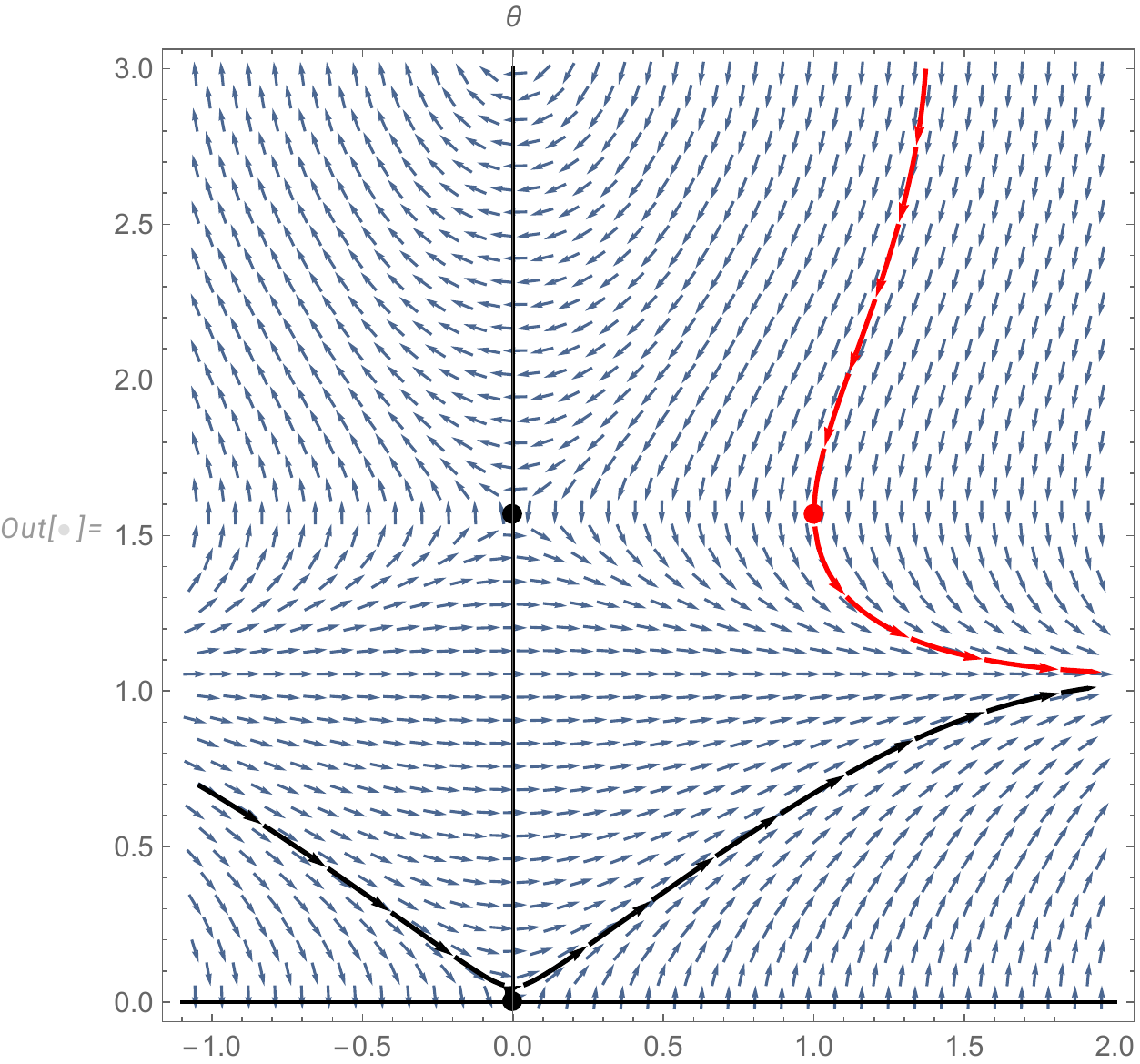} 
\end{center}
\caption{The phase plane of the system \eqref{au} with trajectories emanating from $(0,0)$ (black)  and $(1,\frac{\pi}{2})$ (red).The singular point is $(0,\frac{\pi}{2}) $. Left: case $\lambda>0$.  Right: case $\lambda\in (-1,0)$.   }\label{figphase}
\end{figure}
Since we are considering solutions of \eqref{trig} where $\theta'(0)>0$, the trajectory of \eqref{au} corresponding to the solution  \eqref{trig}-\eqref{trig2} moves on $\theta>0$. As we see in Fig. \ref{figphase}, this trajectory turns around the center without to arrive to the value $\theta=\pi$. Thus the trajectory intersects the axis $x=0$ in some value $\theta_0\in (\pi/2,\pi)$. This implies that the solution $\gamma$ of \eqref{trig}-\eqref{trig2} meets the $z$-axis at some point $s_1$ with  $\theta(s_1)\in (\pi/2,\pi)$. In particular, the maximal domain of existence of \eqref{trig}-\eqref{trig2} is finite and $\omega<\infty$.

\item Case $\lambda=0$. This situation is known and described in  \cite{al2}. From Eq. \eqref{fi}, the maximal domain is $[0,\infty)$ with $\lim_{s\to\infty}x(s)=\sqrt{2}$. The function $\theta$ has sign in its domain which implies that $\gamma$ is a convex curve.
\item Case $\lambda\in (-1,0)$. Instead to analyze the corresponding phase plane as in the case $\lambda>0$, the results can be obtained directly from \eqref{trig}. See the phase plane in Fig. \ref{figphase}. The function $\theta$ cannot attain the value $\pi/2$. Indeed, if $s_1$ is the first point where $\theta(s_1)=\pi/2$, then we would have $\theta'(s_1)\geq 0$. However from \eqref{trig} we have $\theta'(s_1)=\lambda x(s_1)<0$, which it is contradictory. The same argument proves that the function $\theta(s)$ cannot go to the value $\pi/2$, and thus the function $\theta(s)$ is away from the value $\pi/2$. 

The image of $\theta(s)$ remains in the interval $(0,\pi/2)$ after $s=0$. Indeed, if $\theta(s)$ decreases attaining the value $0$ at $s_2$, then $\theta'(s_2)\leq 0$. However, from \eqref{trig} we have $\lim_{s\to s_2}\theta'(s)=\infty$.  Similarly as in the previous discussion about the value $\pi/2$, the function $\theta(s)$ cannot go to value $0$ because in such a case, $\lim \theta'(s)=\infty$ as $\theta(s)\to 0$. As a consequence, if we fix $\bar{s}$ close $0$, there are  $\theta_1,\theta_2\in\r$, with $ 0<\theta_1<\theta_2<\pi/2$ such that $\theta_1\leq\theta(s)\leq\theta_2$ for all $s>\bar{s}$. This proves $\omega=\infty$ and $\lim_{s\to\infty}x(s)=\lim_{s\to\infty}z(s)=\infty$.  

We now prove the last statement. By dividing \eqref{fi} by $x(s)^2$ and letting $s\to \infty$, the left hand-side of \eqref{fi} is $0$. Then   the L'H\^{o}pital's rule gives
$$-\frac12=\lim_{s\to\infty}\frac{\lambda}{x(s)^2}\int_0^t x(t)\, dt=\frac{\lambda}{2\lim_{s\to\infty}x'(s)}.$$
This gives $\lim_{s\to\infty}\cos\theta(s)=-\lambda$.
 \end{enumerate}
The last statement of the theorem is consequence that $\theta'$ cannot vanish, hence $\theta'$ is always is positive. This implies that the curvature of the planar curve $\gamma$ has sign, in this case, is positive because $\theta'(0)>0$. This is equivalent to say that $\gamma$ is convex.  
\end{proof}

\begin{figure}[hbtp]
\begin{center}
\includegraphics[width=.26\textwidth]{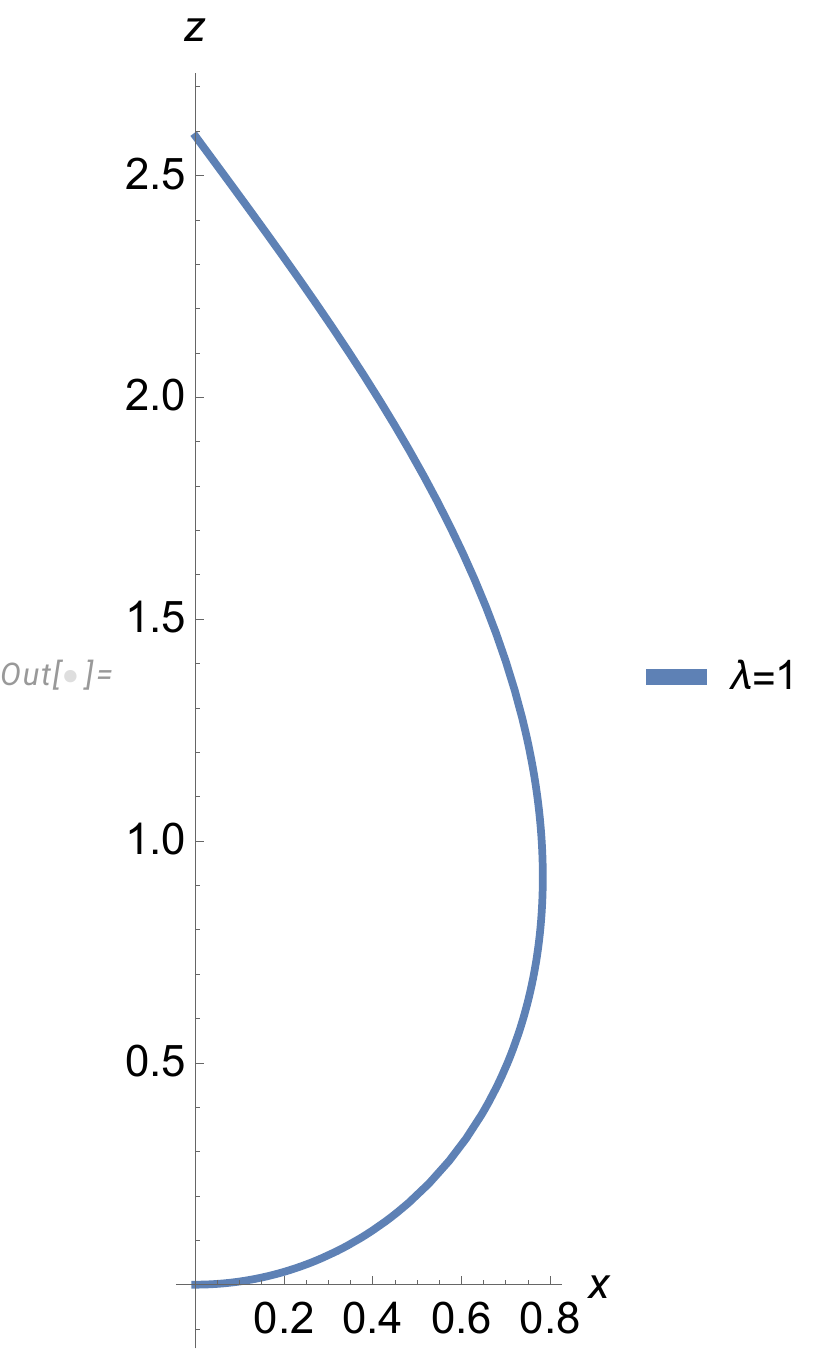}   \includegraphics[width=.28\textwidth]{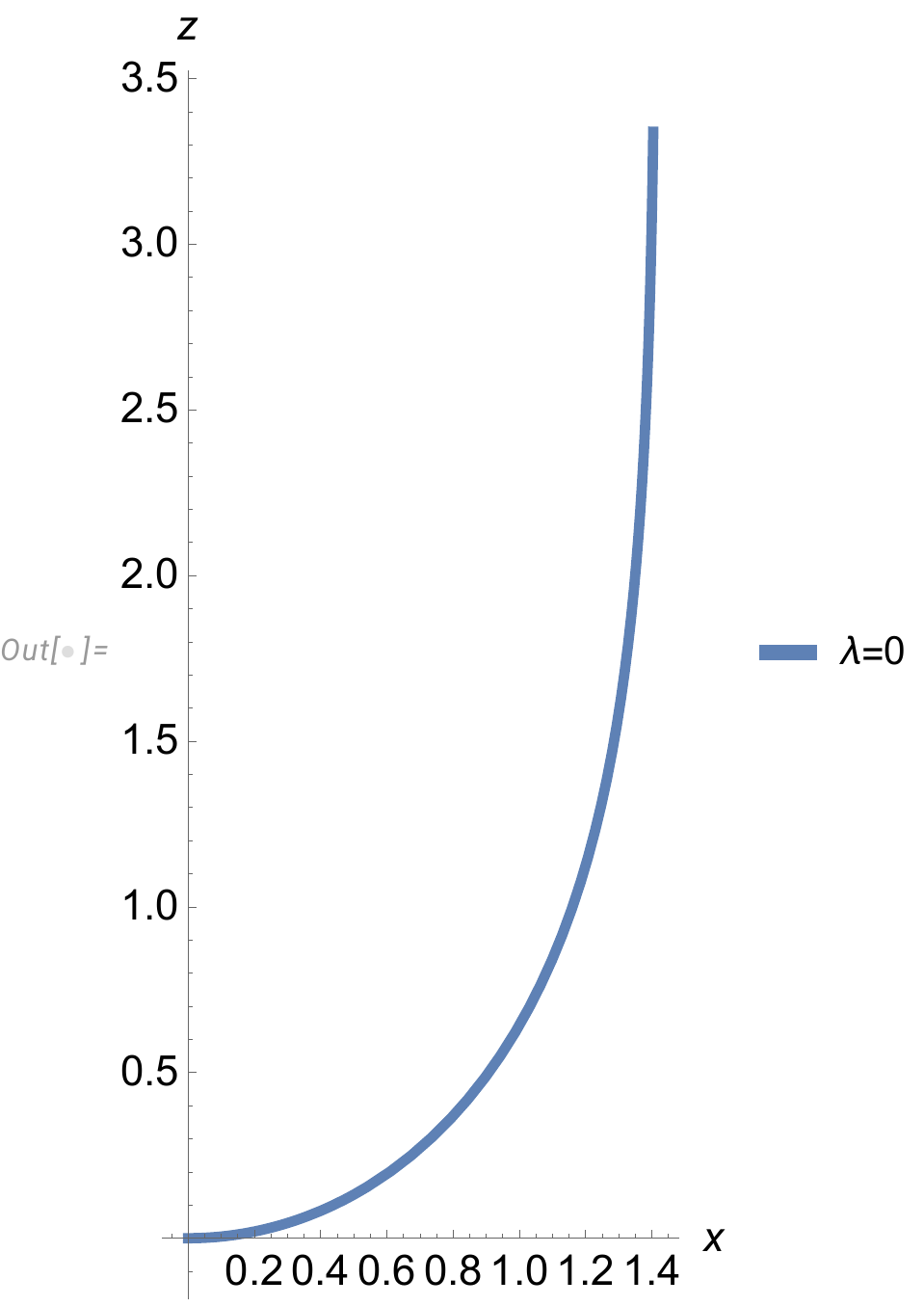}    \includegraphics[width=.4\textwidth]{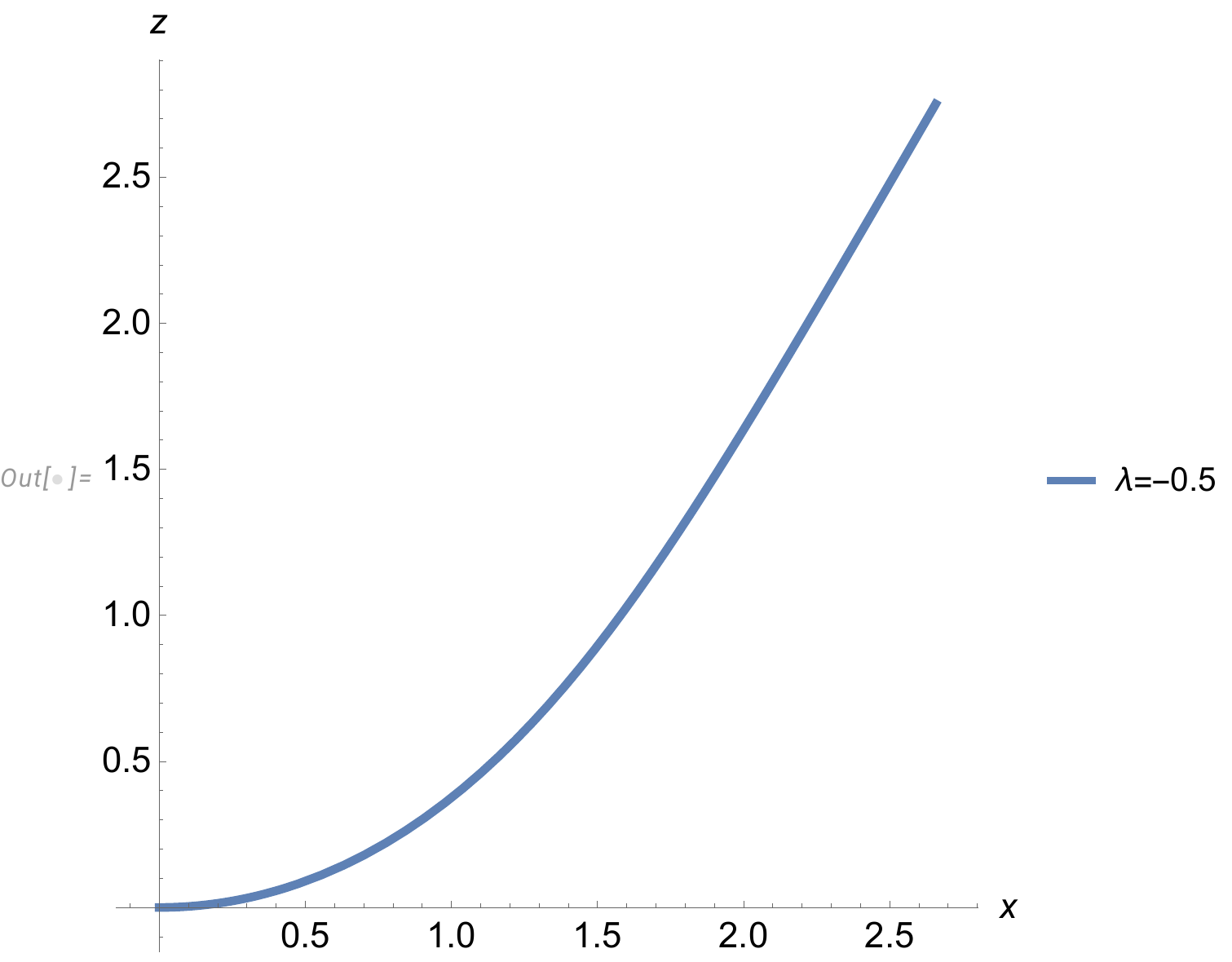} 
\end{center}
\caption{Rotational $\lambda$-translators intersecting the rotation axis for different values of $\lambda$:  $\lambda=1$ (left), $\lambda=0$ (middle) and $\lambda=-0.5$ (right). }\label{fig1}
\end{figure}

Theorem \ref{t2} considers the case that the generating curve $\gamma$ intersects orthogonally the rotation axis   at the starting point $s=0$. With this type of surfaces, it is possible to expect to   have closed $\lambda$-translators if the curve $\gamma$ meets again the rotation axis and at that point, the  intersection is orthogonal.  However, Theorem \ref{t2} asserts that only in the case $\lambda>0$, the curve $\gamma$ comes back to meet the rotation axis, but this intersection is not orthogonal. In consequence, we have proved the non-existence of rotational closed $\lambda$-translators. 

In case of existence of (non-rotational) closed $\lambda$-translators,  we have a restriction on the sign of $\lambda$ and on the genus of the surface.  

\begin{theorem}\label{t-gb} If $\Sigma$ is a closed $\lambda$-translator, then $\lambda\geq 0$ and the genus of $\Sigma$ is $0$ or $1$. Moreover, $\lambda=0$ occurs only if the genus of $\Sigma$ is $1$.

\end{theorem}
\begin{proof}  
Consider a plane $\Pi$, parallel to $\vec{v}$,  such that $\Pi \cap \Sigma = \emptyset$. Moving the plane $\Pi$ parallel to itself towards $\Sigma$, let $p\in\Sigma$ the first point of contact. Then $K(p)\geq 0$ and $N(p)$ is orthogonal to $\vec{v}$. From Eq. \eqref{eq1}, it follows that $\lambda\geq 0$. 

On the other hand, if $g$ is the genus of $\Sigma$, the Gauss-Bonnet theorem, in combination with \eqref{eq1}, gives
$$
4\pi(1-g)=\int_\Sigma K\, d\Sigma=\int_\Sigma \langle N , \vec{v}  \rangle \, d\Sigma+   \lambda\mbox{ area}(\Sigma).$$
Since for any closed surface, we have $\int_\Sigma \langle N , \vec{v}  \rangle d\Sigma=0$, it follows that $4\pi(1-g)=  \lambda\mbox{ area}(\Sigma)$. Since $\lambda\geq 0$, this proves $g\leq 1$. In particular,   $\lambda=0$ if and only if $g=1$. 
\end{proof}

\begin{remark} If the genus is $0$, the value of $\lambda$ is determined by the Gauss-Bonnet formula, namely,  
$$\lambda=\frac{4\pi}{\mbox{area}(\Sigma)}.$$
\end{remark}

When  $\lambda=-1$, Theorem \ref{t1} establishes that if the generating curve meets orthogonally the rotating axis, then the surface is a horizontal plane. In the following result we analyze the case $\lambda=-1$ but now the surface   meets orthogonally the $xy$-plane. In terms of the generating curve $\gamma$, we have     $\theta(0)=\pi/2$ and $x(0)>0$. After a dilation, we can suppose $x(0)=1$. Consider initial conditions
\begin{equation}\label{trig3}
x(0)=1, \quad z(0)=0, \quad\theta(0)=\frac{\pi}{2}.
\end{equation}

\begin{theorem} \label{t5}
Let $\lambda=-1$. The solution $\gamma$ of \eqref{trig}-\eqref{trig3} is a graph on the $z$-axis. Moreover, $\gamma$ has a branch asymptotic to a horizontal line and the other one ends at finite time. 
\end{theorem}

\begin{proof} At $s=0$, from \eqref{trig} it follows $\theta'(0)=-1$, hence $\theta$ is a decreasing function. As $\theta$ belongs to the interval $(0,\pi/2)$, the function $\theta'$ is negative and $\theta$ is decreasing. However, the function $\theta$ cannot attain the value $0$, neither the value $\pi$ by the degeneracy of \eqref{trig}. This proves that $\gamma$ is a graph on the $z$-axis. At $s=0$, we have $x'(0)=0$ and $x''(0)=1$, proving that the function $x(s)$ has a local minimum at $s=0$. This proves that $x(s)\geq 1$ for all $s\in I$. 

The phase plane of the system \eqref{au} appears in Fig. \eqref{fig2}. The linearization of \eqref{au}  at $P=(0,\frac{\pi}{2})$ has a double positive eigenvalue. This implies that singular point $P$ is an unstable node. A branch of the trajectory is asymptotic to the $x$-axis and the other one ends at finite time when $\theta$ attains the value $\pi$. This proves that $\theta$ takes all values in the interval $(0,\pi)$ and a branch of the solution $\gamma$ satisfies $x(s)\to\infty$ as $\theta(s)\to 0$.
\end{proof} 

\begin{figure}[hbtp]
\begin{center}
\includegraphics[width=.34\textwidth]{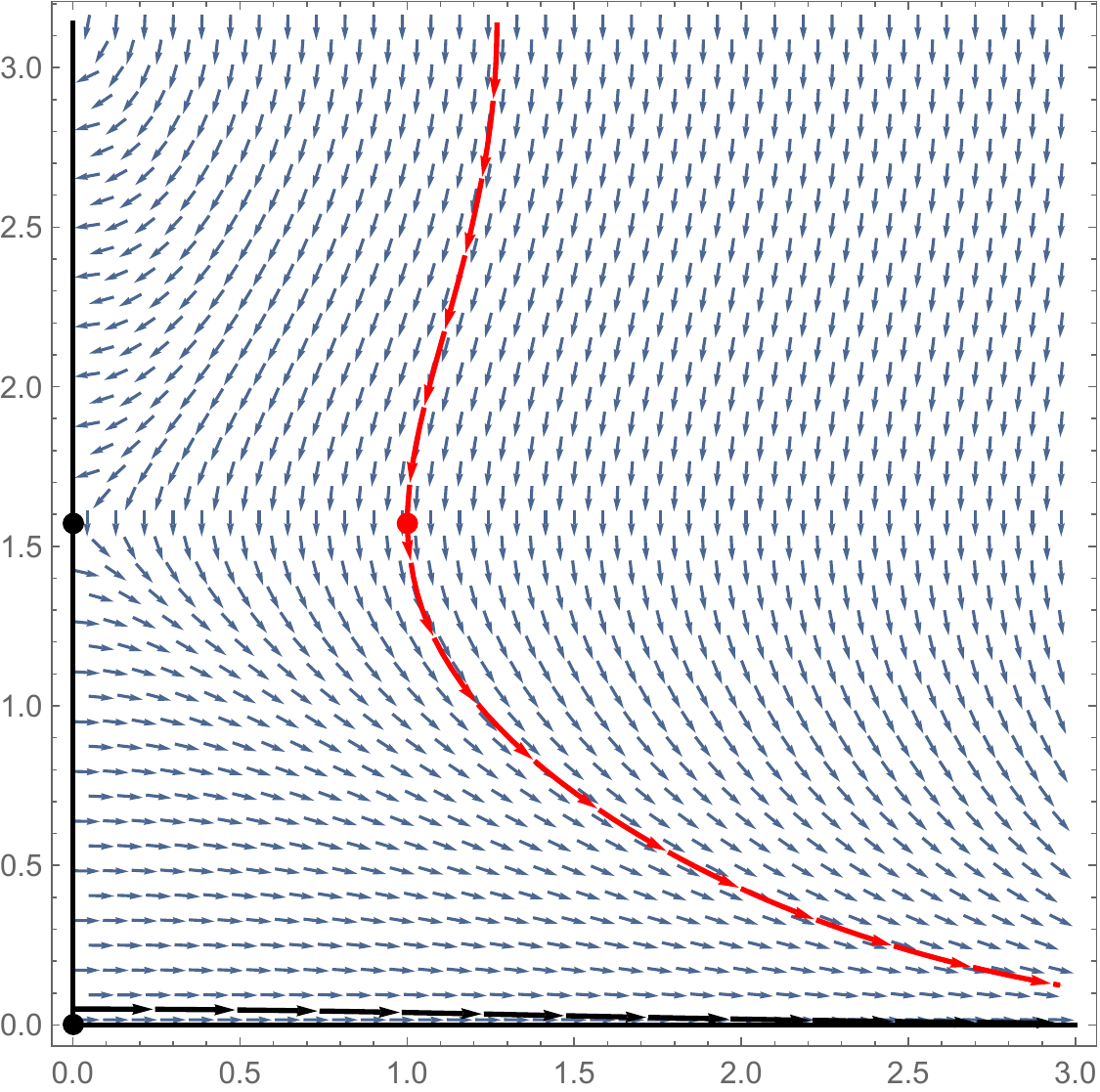} \quad  \includegraphics[width=.6\textwidth]{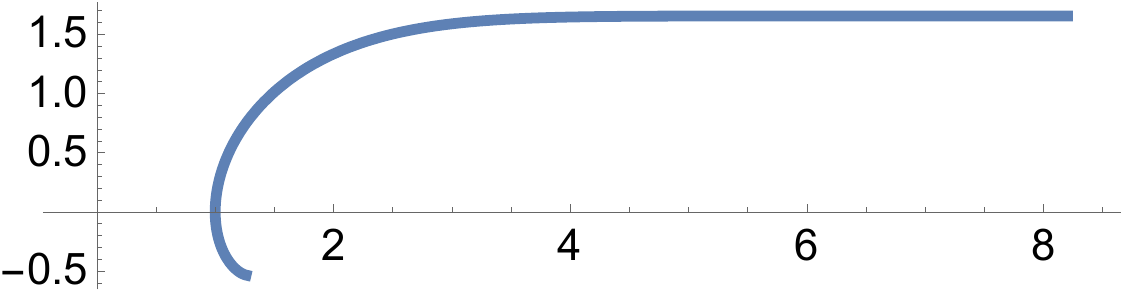}   
\end{center}
\caption{Case $\lambda=-1$. Left: the phase plane. Right: the solution  of \eqref{trig}-\eqref{trig3} }\label{fig2}
\end{figure}

To finish this section, we focus in the case $\lambda<-1$. Then, we know by Thm. \ref{t1}, that  a rotational $\lambda$-translator  cannot  meet orthogonally the rotation axis. However, it is possible to find rotational $\lambda$-translators by dropping this condition.    As in the case $\lambda=-1$, an interesting example occurs with the surface meets orthogonally the $xy$-plane. 
In general, the study of  solutions of \eqref{trig}-\eqref{trig3} can be also considered for all values $\lambda$. In Fig. \ref{figphase} we plotted the (red) trajectories crossing the point $(x,\theta)=(1,\frac{\pi}{2})$ for values $\lambda>0$ and $\lambda\in (-1,0)$. The corresponding   solutions of these trajectories  satisfy the initial conditions \eqref{trig3}. A simple analysis of  Fig. \ref{fig3} gives the following conclusions.
\begin{enumerate}
\item Case $\lambda>0$. A branch of the trajectory ends   when $\theta$ attains the value $0$, where we know that \eqref{trig} degenerates. The other branch of the trajectory ends when $x$ attains the value $0$, which corresponds to the intersection with the rotation axis. 
\item Case $\lambda\in (-1,0)$. A branch of the solution is asymptotic to the angle $\theta_0$ of the case (3) of Thm. \ref{t2}, but the other branch is finite in time when $\theta=\pi$. 
\end{enumerate}
\begin{figure}[hbtp]
\begin{center}
\includegraphics[width=.3\textwidth]{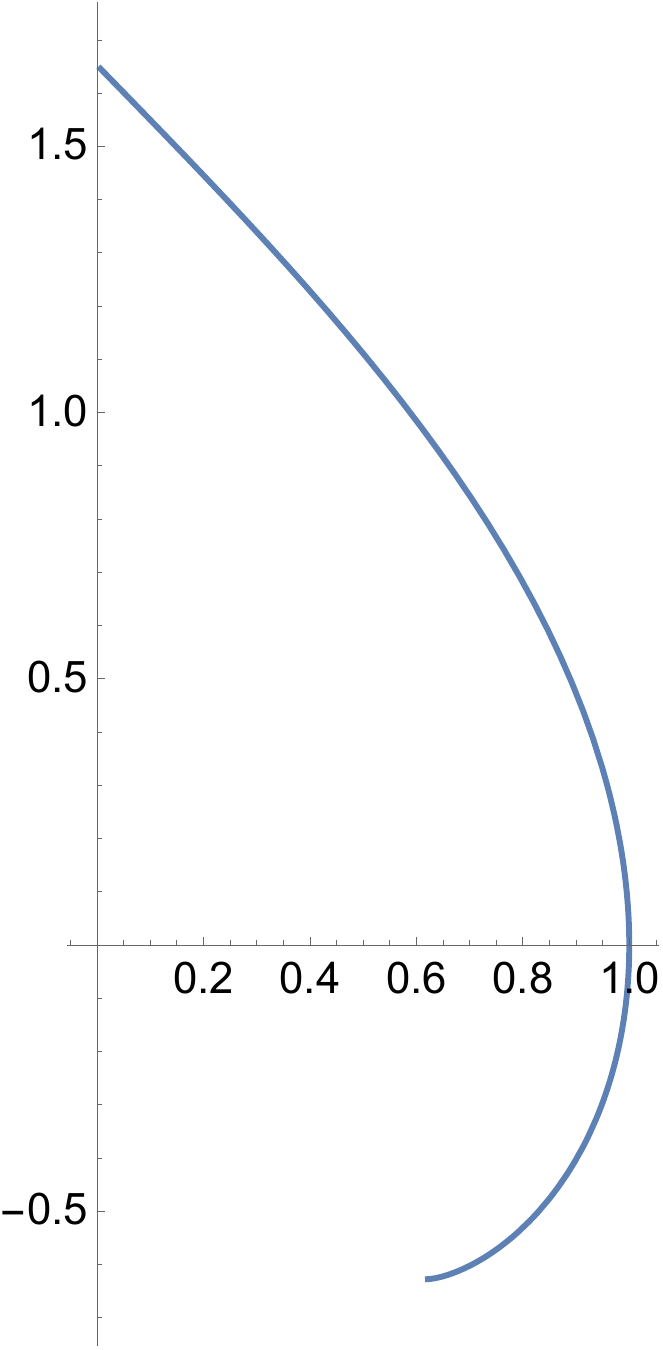} \qquad  \qquad\includegraphics[width=.3\textwidth]{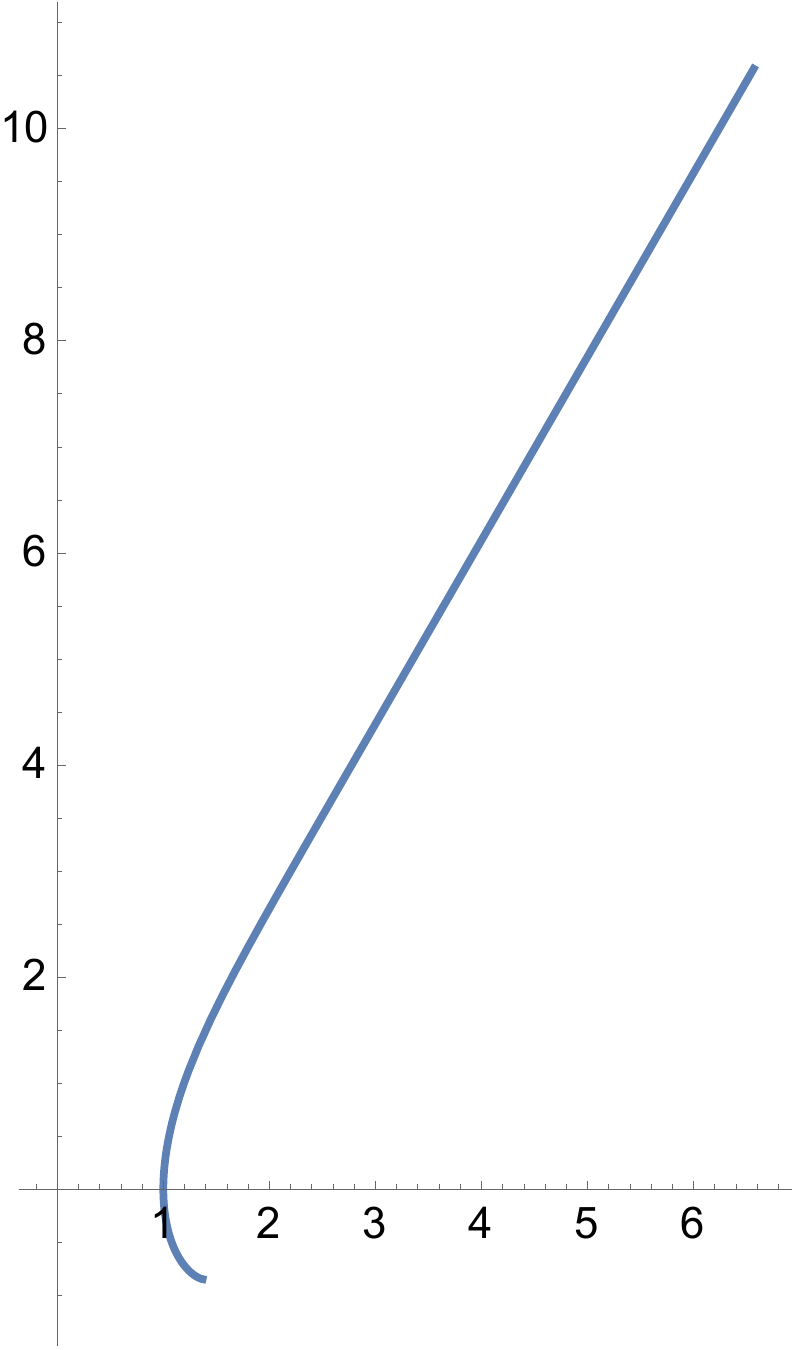}   
\end{center}
\caption{Solutions of \eqref{trig}-\eqref{trig3}:  $\lambda>0$ (left), $\lambda\in (-1,0)$ (right). }\label{fig3}
\end{figure}

Suppose now $\lambda<-1$. At $s=0$ we have $\theta'(0)=\lambda$. In consequence, if $\lambda<0$, the function $\theta$ is decreasing and thus,  that around the point $\gamma(0)$, the surface has negative Gauss curvature. Again, the phase plane \eqref{au}  shows the behaviour of the solution: see Fig. \ref{fig4}.   Now the singular point $P=(0,\frac{\pi}{2})$ is a saddle point because the eigenvalues of the linearization of \eqref{au} are two real numbers with opposite sign. In the phase plane, the trajectory runs from $\theta=\pi$ to $\theta=0$, which implies that the domain of \eqref{trig}-\eqref{trig3} is finite. Moreover, the function $x(s)$ of the solution curve $\gamma$  is bounded. 

\begin{figure}[hbtp]
\begin{center}
\includegraphics[width=.4\textwidth]{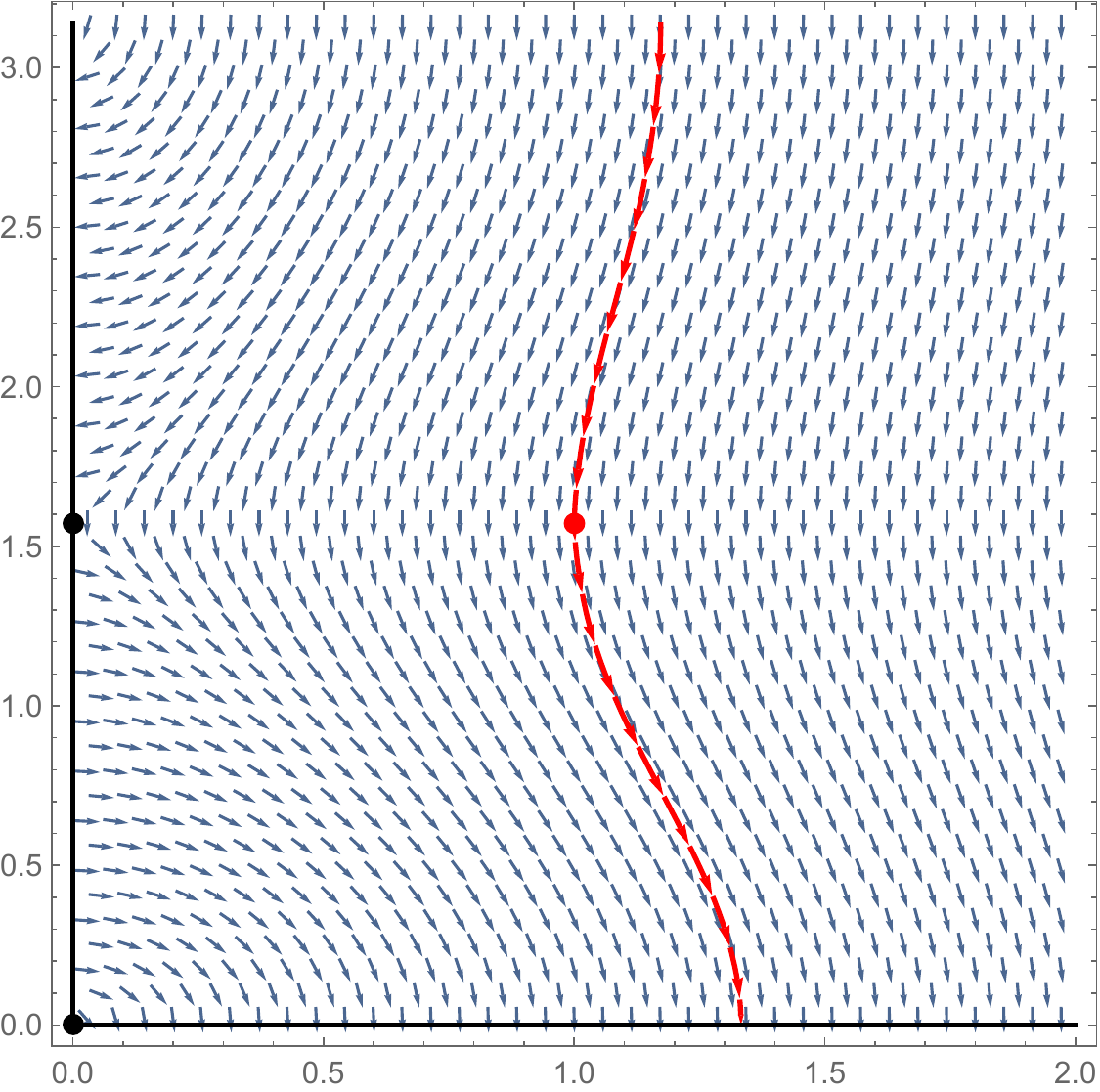} \quad  \includegraphics[width=.5\textwidth]{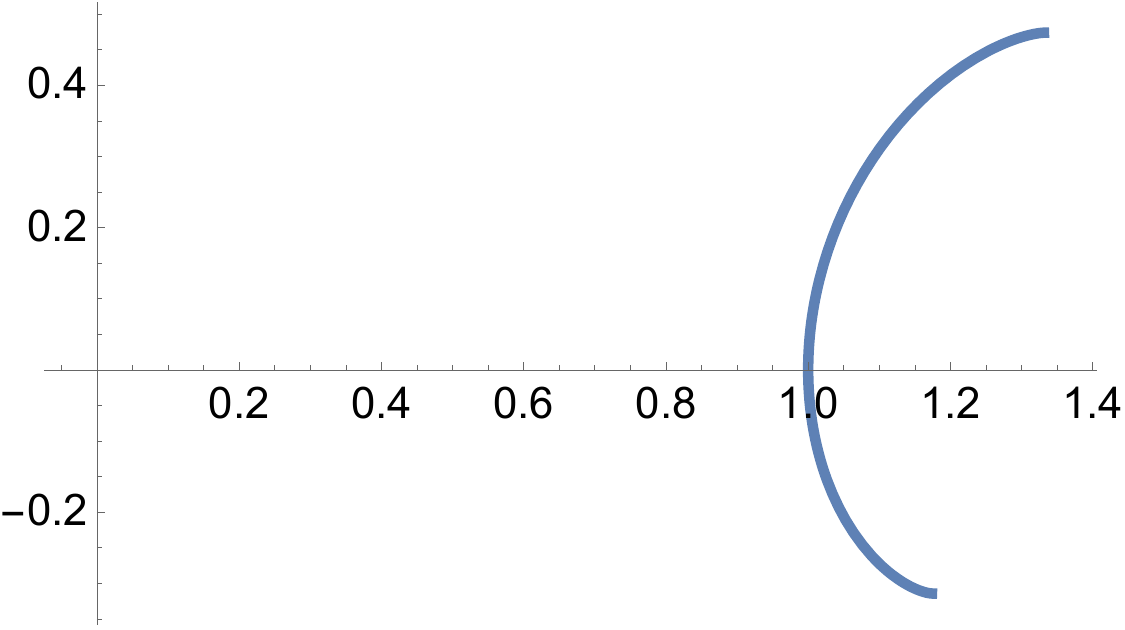}   
\end{center}
\caption{Case $\lambda<-1$. Left: the phase plane. Right: the solution  of \eqref{trig}-\eqref{trig3} }\label{fig4}
\end{figure}

\section*{Acknowledgements}
Rafael L\'opez has been partially supported by MINECO/MICINN /FEDER grant no. PID2023-150727NB-I00,  and by the ``Mar\'{\i}a de Maeztu'' Excellence Unit IMAG, reference CEX2020-001105- M, funded by MCINN/ AEI/10.13039/ 501100011033/ CEX2020-001105-M.

A part of this work was written during R. L\'opez's visit to Fırat University, Turkey, between March and April 2025. The visit was supported by the Scientific and Technological Research Council of Turkey (TÜBİTAK) under its Fellowship Program for Visiting Scientists and Scientists on Sabbatical Leave.



\begin{thebibliography}{00}

\bibitem{an} B. Andrews, Contraction of convex hypersurfaces in Euclidean space. Calc. Var. Partial Differ. Equation 2  (1994), 151--171.


\bibitem{an3}  B. Andrews, Gauss curvature flow: the fate of the rolling stones. Invent. Math. 138  (1999), 151--161.



\bibitem{al}  M. E. Aydin, R. L\'opez, Translators of flows by powers of the Gauss curvature. Ann. Mat. Pura Appl. 202 (2023), 235–251.

\bibitem{al2}  M. E. Aydin, R. L\'opez, Rotational $K^\alpha$-translators in Minkowski space. Taiwanese J. Math. 27(5) (2023), 953–969.

\bibitem{cd} P. Cermelli, A. J. Di Scala, Constant-angle surfaces in liquid crystals. Philosophical Magazine, 87 (2007), 1871-1888.

\bibitem{chu} M. Chen, J. Huang, Flow by powers of the Gauss curvature in space forms. Adv. Math. 442 (2024), 109579.

\bibitem{ch} B. Chow, Deforming convex hypersurfaces by the nth root of the Gaussian curvature.  J. Differential Geom. 22   (1985), 117-138.

\bibitem{cdl}  K. Choi, P. Daskalopoulos, K.-A. Lee, Translating solutions to the Gauss curvature flow with flat sides. Analysis $\&$ PDE 14  (2021), 595-616.

\bibitem{cdkl}  K. Choi, P. Daskalopoulos, L. Kim, K.-A. Lee, The evolution of complete non-compact graphs by powers of Gauss curvature. J. Reine Angew. Math. 757 (2019), 131–158.

\bibitem{ccd} B. Choi, K. Choi, P. Daskalopoulos, Uniqueness of ancient solutions to Gauss curvature flow asymptotic to a cylinder. J. Differential Geom. 127(1) (2024), 77-104.

\bibitem{cck} B. Choi, K. Choi, S. Kim, Continuous family of surfaces translating by powers of Gauss curvature. arXiv:2402.17075v2.

\bibitem{dfvv} F. Dillen, J. Fastenakels, J. Van der Veken, L. Vrancken, Constant angle surfaces in $S^2 \times R$. Monaths. Math. 152 (2007), 89--96.

\bibitem{dsr} A. J. Di Scala, G. Ruiz-Hern\'andez, Helix submanifolds of Euclidean spaces. Monasth Math. 157 (2009), 205--215.



\bibitem{fi} W. J. Firey, Shapes of worn stones. Mathematika 21 (1974), 1--11.



 

\bibitem{jbj} H. Ju, J. Bao, H. Jian, Existence for translating solutions of Gauss curvature flow on exterior domains. Nonlinear Anal. 75 (2012),  3629--3640.

\bibitem{lo3} R. L\'opez, Rigidity of solitons of the Gauss curvature flow in Euclidean space. Preprint (2025).

\bibitem{loy} P. Lucas, J. A. Ortega-Yagües, Slant helices in the Euclidean 3-space revisited. Bull. Belg. Math. Soc. Simon Stevin 23(1) (2016), 133–150.

\bibitem{mn} M. I. Munteanu, A.-I. Nistor, A new approach on constant angle surfaces in $E^3$. Turkish J. Math. 33 (2009), 169–178.

\bibitem{ts} K. Tso,  Deforming a hypersurface by its Gauss-Kronecker curvature. Comm. Pure Appl. Math. 38 (1985), 867--882.

\bibitem{ur} J. Urbas, An expansion of convex hypersurfaces. J. Differential Geom. 33 (1991), 91--125.

\bibitem{ur2} J. Urbas, Complete noncompact self-similar solutions of Gauss curvature flows, I: Positive powers. Math. Ann. 311  (1998), 251--274.

\bibitem{ur3} J. Urbas, Complete noncompact self-similar solutions of Gauss curvature flows, II: Negative powers. Adv. Differ. Equ. 4(3)  (1999), 323--346.


\end{thebibliography}
\end{document}